\documentclass[12pt]{article}
\usepackage{amsfonts}
\usepackage{bbm}

\usepackage{amsmath,mathrsfs,amsfonts,amsthm,amssymb,latexsym,bm}

\usepackage{graphicx}

\newtheorem{thm}{Theorem}[section]

\newtheorem{lem}[thm]{Lemma}
\newtheorem{cor}[thm]{Corollary}
\newtheorem{re}[thm]{Remark}

\begin{document}
\title{\Large \bf Randi\'c Energy and Randi\'c Eigenvalues\footnote{Supported by
NSFC No. 11101232 and 11371205, and China Postdoctoral Science
Foundation. $\dag$ Corresponding author. }}
\author{
{\bf Xueliang Li$^a$, Jianfeng Wang$^{a,b,\dag}$}\\[3mm]
\small $^{a}$Center for Combinatorics and LPMC-TJKLC\\
\small Nankai University 300071, Tianjin, China\\
\small $^{b}$Department of Mathematics\\
\small Qinghai Normal
University 810008, Xining, Qinghai, China\\
\small Email: lxl@nankai.edu.cn; \ jfwang@aliyun.com}
\date{ }

\maketitle

\begin{abstract}
Let $G$ be a graph of order $n$, and $d_i$ the degree of a vertex
$v_i$ of $G$. The {\it Randi\'c matrix} ${\bf R}=(r_{ij})$ of $G$ is
defined by $r_{ij} = 1 / \sqrt{d_jd_j}$ if the vertices $v_i$ and
$v_j$ are adjacent in $G$ and $r_{ij}=0$ otherwise. The {\it
normalized signless Laplacian matrix $\mathcal{Q}$} is defined as
$\mathcal{Q} =I+\bf{R}$, where $I$ is the identity matrix. The {\it
Randi\'c energy} is the sum of absolute values of the eigenvalues of
$\bf{R}$. In this paper, we find a relation between the normalized
signless Laplacian eigenvalues of $G$ and the Randi\'c energy of its
subdivided graph $S(G)$. We also give a necessary and sufficient
condition for a graph to have exactly $k$ and distinct Randi\'c
eigenvalues.
\end{abstract}
\textbf{\emph{}}
\section{\large Introduction}

All graphs considered here are simple, undirected and finite. Let
$G$ be a graph with vertex set $V(G) = \{v_1,v_2,\cdots,v_n\}$ and
degree sequence $(d_1,d_2,\cdots,d_n$), where $d_i$ is the degree of
a vertex $v_i$ ($1\leq i \leq n$) of $G$. For a graph $G$, let $M =
M(G)$ be a corresponding {\it graph matrix} defined in a prescribed
way. The {\it $M$-polynomial} of $G$ is defined as
$\phi_M(G,\lambda)$= {\rm det}$(\lambda{I}-M)$, where $I$ is the
identity matrix. The {\it $M$-eigenvalues} of $G$ are those of its
graph matrix $M$. It is well-known that there already exist some
graph matrices, including {\it adjacency matrix $A$, degree matrix
$D$, Laplacian matrix $L = D-A$, signless Laplacian matrix} $Q=D+A$
and so on.

In 1975, Milan Randi\'c \cite{radic-index} invented a molecular
structure descriptor defined as
$$R(G)=\sum_{i\sim j}\frac{1}{\sqrt{d_id_j}},$$ where the summation goes
over all pairs of adjacent vertices of the underlying (molecular) graph.
This graph invariant is nowadays known under the name {\it Randi\'c index},
for details see \cite{gutman-Furtula,li-gutman,li-shi,radic-history}.

Gutman et al. \cite{gutman-Furtula-bozkurt} pointed out that the
Randi\'c-index-concept is purposeful to produce a graph matrix of
order $n$, named {\it Randi\'c matrix} ${\bf R}(G)$, whose
$(i,j)$-entry is defined as
$$r_{ij} =
\left\{\begin{array}{ccc}
\frac{1}{\sqrt{d_id_j}} &\mbox{if $v_j$ and $v_j$ are adjacent vertices,}\phantom{bbbbbbbb}\\[1mm]
0\phantom{bb} &\mbox{if the vertices $v_j$ and $v_j$ are not adjacent,}\\[1mm]
0\phantom{bb} &\mbox{if $i=j$.\phantom{bbbbbbbbbbbbbbbbbbbbbbbbbb}}
\end{array}
\right.
$$

In what follows, we need the convention that all graphs possess no
isolated vertices. Then ${\bf R}(G) =
D^{-\frac{1}{2}}AD^{-\frac{1}{2}}$. Recall that the {\it normalized
Laplacian}  and {\it sinless Laplacian matrices} \cite{chung} are
respectively defined as
$$\mathcal{L}(G) = D^{-\frac{1}{2}}LD^{-\frac{1}{2}} \quad \mbox{and}
\quad  \mathcal{Q}(G)=D^{-\frac{1}{2}}QD^{-\frac{1}{2}}.$$ From this
point of view, the eigenvalues of above three matrices have a direct
relation. As shown in \cite{gutman-Furtula-bozkurt}, $\mathcal{L}(G)
= I_n-{\bf R}(G)$ and $\mathcal{Q}(G) = I_n+{\bf R}(G)$. So if an
$\bf{R}$-eigenvalue is $\rho_i$, then the $\mathcal{L}$-eigenvalue
$\mu_i$ and $\mathcal{Q}$-eigenvalue $\theta_i$ are respectively
\begin{equation}\label{R-Q-L-eigen}
\mu_i = 1-\rho_i \quad \mbox{and} \quad \theta_i = 1+\rho_i, \;\; 1 \leq i \leq n.
\end{equation}
For the $\mathcal{L}$-eigenvalues,
there are numerous results; see \cite{chung} for example. From Lemmas 1.7--1.8 \cite{chung}
it follows that $0 \leq \mu_i \leq 2$, and so by \eqref{R-Q-L-eigen},
\begin{equation}\label{-10rho-mu12}
-1 \leq \rho_i \leq 1 \quad \mbox{and} \quad 0 \leq \theta_i \leq  2,\;\; 1 \leq i \leq n.
\end{equation}

Gutmam \cite{gutman1} introduced the (adjacency) energy of a graph
$G$ as follows
$$E(G) = \sum_{i=1}^n|\lambda_i|,$$ which has been extended to energies of other
graph matrices \cite{li-shi-gutman,niki}. Especially, the {\it
Randi\'c energy $RE(G)$}
\cite{bozkurt-gungor-gutman1,bozkurt-gungor-gutman2} is defined as
$$RE(G) = \sum_{i=1}^n|\rho_i|.$$

So far, there are quite a few results about the Randi\'c energy and
$\bf{R}$-eigenvalues, which therefore becomes the main research
objects of this paper. In the rest of the paper, we will give a
relation between the $\mathcal{Q}$-eigenvalues of a graph and the
Randi\'c energy of its subdivision in Section 2. We also give a
necessary and sufficient condition for a graph to have exactly $k$
and distinct $R$-eigenvalues in Section 3, particularly for $k=2,3$.

\section{Randi\'c energy and $\mathcal{Q}$-eigenvalues}

Let $S(G)$ be the subdivision of a graph $G$ that is obtained by
adding a new vertex into each edge of $G$. Evidently, $S(G)$ is a
bipartite graph, and so $V(S(G)) = V_1 \cup V(G)$, where $V_1$ is
the set of new added vertices of degree two.

The following lemma
from matrix theory can be found in, for example, \cite{cve} p. 62.

\begin{lem}\label{matrix}
If $M$ is a nonsingular square matrix, then
$$\begin{matrix}\begin{vmatrix}
M & N \\ P & F
\end{vmatrix}\end{matrix}
= |M| \cdot |F-PM^{-1}N|.$$
\end{lem}

\begin{lem}\label{QG-ASG}
Let $G$ be a graph with order $n$ and size $m$.
Then $$\phi_{_{\bf R}}(S(G),\lambda) = \frac{\lambda^{m-n}}{2^n}\phi_{_{\mathcal{Q}}}(G,2\lambda^2).$$
\end{lem}
\begin{proof}[\bf Proof]
Obviously, $|V(S(G))| = n+m$. It is well-known that $$Q = BB^T \quad \mbox{and} \quad
A(S(G)) = \begin{matrix}
             \begin{pmatrix}
              O & B^T \\
              B & O
             \end{pmatrix}
          \end{matrix}
,$$ where $B$ is the incident matrix of $G$ and $B^T$ is the
transpose of $B$. Then, we partition the degree matrix $D(S(G))$ into
$$D(S(G)) = \begin{matrix}
             \begin{pmatrix}
              D_1^{-\frac{1}{2}} & O\\
              O & D_2^{-\frac{1}{2}}
             \end{pmatrix}
          \end{matrix}
,$$ where $D_1 = {\rm diag}(2,2,\cdots,2)$ with order $m\times m$ and
$D_2 = D(G)$. If $G$ has no isolated vertices,
then so does $S(G)$. Consequently,
\begin{align*}
{\bf R}(S(G)) &= D^{-\frac{1}{2}}A(S(G))D^{-\frac{1}{2}}
           = \begin{matrix}
                \begin{pmatrix}
                  D_1^{-\frac{1}{2}} & O\\
                  O & D_2^{-\frac{1}{2}}
                 \end{pmatrix}
               \end{matrix}
\begin{matrix}
\begin{pmatrix}
O & B^T \\
 B & O
 \end{pmatrix}
\end{matrix}
              \begin{matrix}
                \begin{pmatrix}
                  D_1^{-\frac{1}{2}} & O\\
                  O & D_2^{-\frac{1}{2}}
                 \end{pmatrix}
               \end{matrix}\\
&\phantom{000000000000000000}= \begin{matrix}
\begin{pmatrix}
O & D_1^{-\frac{1}{2}}B^TD_2^{-\frac{1}{2}}\\
D_2^{-\frac{1}{2}}BD_1^{-\frac{1}{2}} & O
 \end{pmatrix}
\end{matrix}.
\end{align*}
By Lemma \ref{matrix} we get
\begin{align*}
\phi_{\bf R}(S(G),\lambda) &= |{\lambda}I_{m+n} - R(S(G))| =
\begin{matrix}
\begin{vmatrix}
{\lambda}I_m & -D_1^{-\frac{1}{2}}B^TD_2^{-\frac{1}{2}} \\
-D_2^{-\frac{1}{2}}B^TD_1^{-\frac{1}{2}} & {\lambda}I_n
\end{vmatrix}
\end{matrix}\\
&= |{\lambda}I_m||\lambda{I_n}-
D_2^{-\frac{1}{2}}BD_1^{-\frac{1}{2}}\frac{I_m}{\lambda}D_1^{-\frac{1}{2}}B^TD_2^{-\frac{1}{2}}|\\
&= \lambda^{m-n}|\lambda^2I_n - \frac{1}{2}D_2^{-\frac{1}{2}}BB^TD_2^{-\frac{1}{2}}|\\
&=\frac{\lambda^{m-n}}{2^n}|2\lambda^2I_n - \mathcal{Q}|\\
&=\frac{\lambda^{m-n}}{2^n}\phi_{_{\mathcal{Q}}}(G,2\lambda^2).
\end{align*}
This finishes the proof.
\end{proof}

\begin{thm}\label{RSG-QG}
Let $G$ be a graph with order $n$ and size $m$.
\begin{itemize}
\item[$\mathrm{(i)}$]
If $\phi_{_{\mathcal{Q}}}(G,\lambda) = \sum_{i=0}^na_i\lambda^{n-i}$,
then $\phi_{\bf R}(S(G),\lambda) = \lambda^{m-n}\sum_{i=0}^n2^{-i}a_i\lambda^{n-i}$.
\item[$\mathrm{(ii)}$]
$\rho$ is an ${\bf R}$-eigenvalue of $S(G)$ if and only if $2\rho^2$
is a $\mathcal{Q}$-eigenvalue of $G$.
\item[$\mathrm{(iii)}$]
Let $\theta_1,\theta_2,\cdots,\theta_n$ be the
$\mathcal{Q}$-eigenvalues of $G$. Then $RE(S(G)) =
\sqrt{2}\sum_{i=1}^n\sqrt{\theta_i}$.
\end{itemize}
\end{thm}
\begin{proof}[\bf Proof]
For (i), by Lemma \ref{QG-ASG} we get
$$\phi_{\bf R}(S(G),\lambda) = \frac{\lambda^{m-n}}{2^n}\phi_{_{\mathcal{Q}}}(G,2\lambda^2)
                     = \frac{\lambda^{m-n}}{2^n}\sum_{i=0}^na_i(\sqrt{2}\lambda)^{2(n-i)}
                     = \lambda^{m-n}\sum_{i=0}^n2^{-i}a_i\lambda^{n-i}.$$
(ii) is an immediate result of Lemma \ref{QG-ASG}. For (iii), from
\eqref{-10rho-mu12} we obtain $\theta_i \geq 0$, and so
$\pm\sqrt{\theta_i/2}$ is an ${\bf R}$-eigenvalue of $S(G)$ by (i).
Thus, $RE(S(G)) = \sqrt{2}\sum_{i=1}^n\sqrt{\theta_i}$.
\end{proof}

\begin{re}
By Theorem \ref{RSG-QG}(i), it becomes easier to compute the
Randi\'c energies of some graphs. As an example, Gutman et al.
\cite{gutman-Furtula-bozkurt} conjectured that the connected graph
with odd order and greatest Randi\'c energy  is the sun, which is
exactly the subdivision of the star $S_n$. Easy to compute
$\phi_{\mathcal{Q}}(S_n,\theta) = \theta(\theta-1)^{n-2}(\theta-2)$.
Hence, $RE(S(S_n)) = \sqrt{2}n+2-2\sqrt{2}$.
\end{re}

\section{Connected graphs with distinct ${\bf R}$-eigenvalues}

A popular and important research field is to investigate the
connected graphs with distinct eigenvalues. As van Dam said, it is
an interplay between combinatorics and algebra; for details see his
thesis \cite{dam-thesis}. Inspired by his ideas, we give a necessary
and sufficient condition for a graph to have $k$ distinct ${\bf
R}$-eigenvalues.

It has been proved that $\rho_1 =1$ is the largest ${\bf
R}$-eigenvalues with the Perron-Frobenius vector $\alpha^T =
(\sqrt{d_1},\sqrt{d_2},\cdots,\sqrt{d_n})$; see
\cite{cavers-fallat-kirkland,gutman-Furtula-bozkurt,liu-huang-feng}.

\begin{thm}\label{k-distinct-R-eigen}
Let $G$ be connected graph with order $n \geq 3$ and size $m$. Then
$G$ has exactly $k$ $(2 \leq k \leq n)$ and distinct ${\bf
R}$-eigenvalues if and only if there are $k-1$ distinct none-one
real numbers $\rho_2, \rho_3, \cdots,\rho_k$ satisfying
\begin{equation}\label{rho=alpha}
\prod_{i=2}^k({\bf R}-\rho_iI) = \frac{\prod_{i=2}^k(1-\rho_i)}{2m}\alpha\alpha^T.
\end{equation}
Moreover, $1,\rho_2,\cdots,\rho_k$ are exactly the $k$ distinct
${\bf R}$-eigenvalues of $G$.
\end{thm}
\begin{proof}[\bf Proof]
Let $\rho_1=1, \rho_2,\rho_3,\cdots,p_k$ be the $k$ distinct ${\bf R}$-eigenvalues.
Since ${\bf R}$ is a real symmetric matrix, it must be diagonalizable, and thus
the minimal polynomial of ${\bf R}$ is $(\lambda-\rho_1)(\lambda-\rho_2)\cdots(\lambda-\rho_k)$.
Hence, $$\Pi_{i=1}^n({\bf R}-\rho_iI) = O, \quad \mbox{that is,} \quad
({\bf R}-\rho_1I)(\Pi_{i=2}^n({\bf R}-\rho_iI)) = O.$$
Since $G$ is connected,  by Perron-Frobenius Theorem we know that the
algebraic multiplicity of $\rho_1 = 1$ is one, and so is the geometric multiplicity.
Consequently, each column of $H=\Pi_{i=2}^n({\bf R}-\rho_iI) = (h_1,h_2,\cdots,h_n)$
is a scalar multiple of the Perron-Frobenius vector $\alpha$.
Set $h_i = a_i\alpha$ ($1 \leq i \leq n$). So, $H= \alpha(a_1,a_2,\cdots,a_n)$
and thus $$\alpha^TH = \alpha^T\alpha(a_1,a_2,\cdots,a_n).$$
By a direct calculation we have
$$\prod_{i=2}^k(1-\rho_i)\alpha^T = 2m(a_1,a_2,\cdots,a_n),$$
leading to $$a_i =
\frac{\prod_{i=2}^k(1-\rho_i)}{2m}\sqrt{d_i}\;(i=1,2,\cdots,k).$$
The necessity thus follows.

For the sufficiency, multiplying ${\bf R}-\rho_1I$ ($\rho_1=1$) to
both sides of \eqref{rho=alpha}, we arrive at $$({\bf
R}-\rho_1I)\prod_{i=2}^k({\bf R}-\rho_iI) =
\frac{\prod_{i=2}^k(1-\rho_i)}{2m}(({\bf R}-I)\alpha)\alpha^T = O.$$
So, $m(x)=(x-\rho_1)(x-\rho_2)\cdots(x-\rho_k)$ is an annihilating
polynomial of ${\bf R}$, i.e., a polynomial with value at ${\bf R}$
is a 0-matrix. From \eqref{rho=alpha} follows $\prod_{i=2}^k({\bf
R}-\rho_iI) \neq O$, which shows that the product of some 
factors taken from $\{x-\rho_2,\cdots,x-\rho_k\}$ is not a minimal 
polynomial of ${\bf R}$. Hence, $m(x)$
is the minimal polynomial, and thus $1,\rho_2,\cdots,\rho_k$ are the
$k$ distinct ${\bf R}$-eigenvalues.
\end{proof}

Bozkurt et al. \cite{bozkurt-gungor-gutman2} determined the
connected graphs with two distinct ${\bf R}$-eigenvalues. We now
give another short proof based on the above theorem.

\begin{cor}
A connected graph $G$ has exactly two and distinct ${\bf
R}$-eigenvalues if and only if $G$ is a complete graph.
\end{cor}
\begin{proof}[\bf Proof]
It is known that the complete graph of order $n$ has exactly two
distinct ${\bf R}$-eigenvalues $1$ and $-\frac{1}{n-1}$ \cite{cve}.
Substituting $1$ and $-\frac{1}{n-1}$ into Eq. \eqref{rho=alpha} we
get
$${\bf R}(G) = \frac{1}{(n-1)^2}\alpha\alpha^T - \frac{1}{n-1}I.$$
Considering the diagonal entries in both sides of the above equality,
we have $$\frac{1}{(n-1)^2}d_i - \frac{1}{n-1} = 0,$$ and so $d_i = n-1$
($i = 1,2,\cdots,n$). Hence, $G$ is a complete graph.
\end{proof}

For the graph with exactly three and distinct ${\bf R}$-eigenvalues,
the following characterization is given. We denote the number of
common neighbors by $\delta_{ij}$ if vertices $v_i$ and $v_j$ are
adjacent, and by $\sigma_{ij}$ if they are not.

\begin{cor}\label{three-R-characterization}
Let $c=\frac{\prod_{i=2}^k(1-\rho_i)}{2m}$. A connected graph $G$
has exactly three and distinct ${\bf R}$-eigenvalues
$1,\rho_2,\rho_3$ if and only if the following items hold:
\begin{itemize}
\item[$\mathrm{(i)}$]
for any vertex $u_i$, $\sum_{v_j \sim u_i}\frac{1}{d_j} = cd_i^2 -
\rho_2\rho_3d_i$,
\item[$\mathrm{(ii)}$]
for adjacent vertices $u_i$ and $v_j$, $\delta_{ij} = cd_id_j +
\rho_2 + \rho_3$,
\item[$\mathrm{(iii)}$]
for nonadjacent vertices $u_i$ and $v_j$, $\sigma_{ij} = cd_id_j$.
\end{itemize}
\end{cor}
\begin{proof}[\bf Proof]
From Theorem \ref{k-distinct-R-eigen} we get $({\bf R}-\rho_2I)({\bf
R}-\rho_3I) = c\alpha\alpha^T.$ Then the results follow by
considering the diagonal entries and nondiagonal entries for both
sides of this equality.
\end{proof}

Note that a $k$-regular graph of order $n$ ($0 < k < n-1$) is
{\it strong regular} with parameters $(n,k,\delta,\sigma)$
if the number of common neighbors of any two distinct vertices equals
$\delta$ if the vertices are adjacent and $\sigma$ otherwise \cite{brouwer-haemers}.
The following result follows from Corollary \ref{three-R-characterization}.

\begin{cor}
A regular connected graph has exactly three and distinct ${\bf
R}$-eigenvalues if and only if it is strong regular.
\end{cor}

From \eqref{R-Q-L-eigen} it follows that a connected graph has
exactly $k$ distinct ${\bf R}$-eigenvalues if and only if it has $k$
distinct $\mathcal{L}$-ones. van Dam and Omidi \cite{dam-omidi}
found such graphs and pointed out that a complete classification of
such graphs still seems out of reach. In subsequent work, it seems
interesting to determine connected graphs with exactly four or more
and distinct ${\bf R}$-eigenvalues. Furthermore, due to $\mathcal{L}
= I - {\bf R}$, it seems much simpler to investigate on this topic
by the ${\bf R}$-eigenvalues.

\end{document}